\documentclass[12pt]{article}

\usepackage{amsmath,amssymb,  latexsym, amsthm, amscd, fancyhdr, fullpage, url}

\usepackage[svgnames]{xcolor}
\usepackage{tikz}

\pgfdeclarelayer{edgelayer}
\pgfdeclarelayer{nodelayer}
\pgfsetlayers{edgelayer,nodelayer,main}

\tikzstyle{none}=[inner sep=0pt]
\definecolor{hexcolor0xf81e1c}{rgb}{0.973,0.118,0.110}
\definecolor{hexcolor0x3c00ff}{rgb}{0.235,0.000,1.000}
\definecolor{hexcolor0x24fe00}{rgb}{0.141,0.996,0.000}

\tikzstyle{whitevertex}=[circle,fill=White,draw=Black, scale = 0.5]
\tikzstyle{vertex}=[circle,fill=White,draw=Black, scale = 0.5]
\tikzstyle{redvertex}=[circle,fill=hexcolor0xf81e1c,draw=Black]
\tikzstyle{bluevertex}=[circle,fill=hexcolor0x3c00ff,draw=Black]
\tikzstyle{greenvertex}=[circle,fill=hexcolor0x24fe00,draw=Black]
\tikzstyle{textbox}=[rectangle,fill=none,draw=none]

\tikzstyle{arc}=[Black, ->]

\newtheorem{theorem}{Theorem}[section]

\newtheorem{corollary}[theorem]{Corollary}
\newtheorem{lemma}[theorem]{Lemma}
\newtheorem{proposition}[theorem]{Proposition}
\newtheorem{observation}[theorem]{Observation}

\begin{document}

\title{Switching $m$-edge-coloured graphs using non-Abelian groups}

\author{Chris Duffy, Gary MacGillivray, Ben Tremblay}
\date{}

\maketitle 

\begin{abstract}
Let $G$ be a graph whose edges are each assigned one of the $m$-colours $1, 2, \ldots, m$,
and let $\Gamma$ be a subgroup of $S_m$.
The operation of switching at a vertex $x$ with respect  $\pi \in \Gamma$ 
permutes the colours of the edges incident with $x$ according to $\pi$.  There is a well-developed theory of switching when $\Gamma$ is Abelian. Much less is known for non-Abelian groups.   In this paper we consider switching with
respect to non-Abelian groups including symmetric, alternating and dihedral groups.
We first consider the question of whether 
there is a sequence of switches using elements of $\Gamma$ that transforms
an $m$-edge-coloured graph $G$ to an $m$-edge coloured graph $H$.
Necessary and sufficient conditions
for the existence of such a sequence
are given for each of the groups being considered.
We then consider the question of whether an
$m$-edge coloured graph can be switched using elements of $\Gamma$ so that the
transformed $m$-edge coloured graph has a vertex $k$-colouring, or a homomorphism to a fixed
$m$-edge coloured graph $H$.
For the  groups just mentioned we  establish
dichotomy theorems for the complexity of
these decision problems.
These are the first dichotomy theorems to be established for colouring or homomorphism problems and switching with 
respect to any group other than $S_2$.
\end{abstract}

\section{Introduction and Definitions}
An $m$-edge-coloured graph is an ordered pair $G = (H, \Sigma)$, 
where $H$ is a graph and $\Sigma: E(H) \to \{1, 2, \ldots, m\}$ is its \emph{signature}.
The graph $H$ is the \emph{underlying graph} of $G$, and may also be referred to as 
\emph{underlying$(G)$}.  
The vertices of $G$ are the vertices of $H$.
The edges of $G$ are coloured edges of $H$, that is, edges $e \in E(H)$ together
with their signature (or \emph{colour}) $\Sigma(e)$.
We use $E_i(G)$ to denote the set of edges of $G$ with colour $i$.
An $m$-edge-coloured graph is \emph{monochromatic of colour $j$} if all of its edges have colour $j$.

Let $G$ be an $m$-edge-coloured graph and let $\Gamma$ be a subgroup of $S_m$.
For $x \in V$ and $\pi \in \Gamma$, the operation of
\emph{switching at $x$ with respect to $\pi$} transforms $G$ into the 
$m$-edge-coloured graph $G^{(x, \pi)}$ that has the same underlying graph as $G$ and with
the colours of the edges incident with $x$ permuted according to $\pi$, that is,
if $\Sigma(G)(xy) = i$, then $\Sigma(G^{(x, \pi)})(xy) = \pi(i)$.

Let $\mathcal{S} = (x_1, \pi_1), (x_2, \pi_2), \ldots, (x_t, \pi_t)$ be a sequence of 
elements of $V(G) \times \Gamma$. 
Recursively define
$$G^\mathcal{S} 
= G^{(x_1, \pi_1), (x_2, \pi_2), \ldots, (x_t, \pi_t)}
= \left(G^{(x_1, \pi_1)}\right)^{(x_2, \pi_2), \ldots, (x_t, \pi_t)}.$$
We call the sequence $\mathcal{S}$ a \emph{$\Gamma$-switching sequence},
and say that it \emph{transforms} $G$ into $G^{\mathcal{S}}$.
Two $m$-edge-coloured graphs $G$ and $H$ are called \emph{$\Gamma$-switch equivalent} if there exists a $\Gamma$-switching sequence $\mathcal{S}$ such that
$G^\mathcal{S} \cong H$.
In other words, $G$ and $H$ are $\Gamma$-switch equivalent if there exists a $\Gamma$-switching sequence that transforms $G$ into an $m$-edge coloured graph that is isomorphic to $H$. 
It is easy to see that $\Gamma$-switch equivalence defines an equivalence relation on the set of all $m$-edge
coloured graphs.  The equivalence class of the $m$-edge-coloured graph $G$ is denoted by $[G]_\Gamma$.

Switching 2-edge coloured graphs with respect to $S_2$ first
appears in the work of Abelson and Rosenberg in the context of 
behavioural science \cite{AR}.  Switching 2-edge coloured graphs in which the colours are $\{+1, -1\}$ is integral to the study of signed graphs.  These are different than 2-edge coloured graphs because the product of colours on each cycle in invariant under switching, which leads to the fundamental concept of \emph{balance} of a cycle.  Signed graphs have been extensively studied by Zaslavsky; for example see \cite{Zaslavsky, ZaslavskySurvey}. 
The related concept of \emph{pushing vertices} in oriented graphs is considered in \cite{KlosMacG}.
Switching $m$-edge coloured graphs with
respect to cyclic groups was first studied by Brewster and Graves \cite{BrewsterGraves}.  Their results are extended to all Abelian groups in \cite{LMW}.

After noting some preliminary information, in the first part of this paper we consider 
the question of when two $m$-edge-coloured graphs $G$ and $H$
are $\Gamma$-switch equivalent when $\Gamma$ is a symmetric, 
alternating or dihedral group, or belongs to a family of other groups.  
In each case we give necessary and sufficient conditions
for two $m$-edge coloured graphs $G$ and $H$ to be
$\Gamma$-switch equivalent.
We believe these to be the first results on switch equivalence
with respect to non-Abelian groups.
For an Abelian group $\Gamma$ and an $m$-edge coloured graph $G$
there is a graph $P_\Gamma(G)$ such that $G$ and $H$ are switch equivalent if and only if $P_\Gamma(G) \cong P_\Gamma(H)$ 
\cite{LMW} (also see \cite{BrewsterGraves,  SenThesis}, and \cite{KlosMacG} for similar results in the context of oriented graphs), if and only if $H$ is a special type of subgraph of $P_\Gamma(G)$
\cite{LMW} (the result is implicit in \cite{BrewsterGraves}).
No similar results are known to hold when $\Gamma$ is non-Abelian.
For switching with respect to $S_2$, another different necessary and sufficient condition has been given by Zaslavsky (see Corollary \ref{ZaslavskyTheorem}, where a similar condition is shown to hold for $D_m$, $m$ even).

In the last part of the paper we consider colourings and 
homomorphisms of $m$-edge coloured graphs. 
We are interested in the complexity of deciding whether a given
$m$-edge coloured graph $G$ can be switched so it has
a vertex $k$-colouring or a homomorphism to a fixed
$m$-edge coloured graph $H$.
We are able to give dichotomy theorems for these problems with respect to the groups we consider.
A dichotomy theorem for $\Gamma$-switchable $k$-colouring 
when $\Gamma$ is Abelian appears in $\cite{LMW}$.
Kidner has proved that for all groups $\Gamma$ the
problem of deciding whether an $m$-edge coloured graph $G$ has a $\Gamma$-switchable $k$-colouring 
is solvable in polynomial time when $k \leq 2$
and is NP-hard when $k \geq 3$ \cite{Kidner} (also see
\cite{BKM, LMW}).
The dichotomy theorems for the homomorphism problem generalize
the fundamental result of Hell and Ne\v{s}et\v{r}il, and the dichotomy theorem for $S_2$ switchable homomorphism due to
Brewster et al. \cite{BFHN}.
Related results for oriented graphs appear in \cite{KlosMacG}.


We note that, for any group $\Gamma$,  two $m$-edge-coloured graphs which are both monochromatic of colour $j$ are $\Gamma$-switch equivalent
if and only if their underlying graphs are isomorphic.
Hence deciding whether two $m$-edge coloured graphs are
switch equivalent is at least as hard as deciding whether 
they are isomorphic.

We now describe a way to determine whether two $m$-edge coloured 
graphs $G$ and $H$ are
$\Gamma$-switch equivalent with respect to any 
group $\Gamma$.  One can
construct an auxiliary graph with vertex set equal to the set of
all $\frac{m!}{|Aut(G)|}|E(G)|^m$ \emph{labelled} $m$-edge coloured graphs, and  
an edge from $F$ to $F'$ if there exists a vertex $x$ of $F$ and a permutation $\pi \in \Gamma$ such that 
$F^{(x, \pi)} = F'$.  
Two (labelled) $m$-edge coloured graphs are $\Gamma$-switch equivalent if and only if they belong to the same component of the auxiliary graph.  
Determining whether $G$ and $H$ are $\Gamma$-switch equivalent
using this procedure involves considering $\Gamma$-switching
sequences of length at most $\frac{m!}{|Aut(G)|}|E(G)|^m$.

Suppose $\Gamma$ is Abelian.
Then the same transformed
graph arises from any rearrangement of a given
$\Gamma$-switching sequence.
Since there is a rearrangement
so that the switches at each vertex occur consecutively, and the result of 
switching with respect to $(x, \alpha_1), (x, \alpha_2), \ldots, (x, \alpha_k)$ is the same as the result of switching with respect to
$(x, \alpha_1 \alpha_2 \ldots \alpha_k)$,
it suffices to consider $\Gamma$-switching sequences in which there is at most one switch at each vertex.
Such a sequence has length at most $|V(G)|$.

To see that the order of switches can matter when $\Gamma$ 
is non-Abelian, consider a 3-edge-coloured graph $G$, the group $\Gamma = S_3$, and an edge $xy$ of colour 1.
Let $\alpha = (1\ 2)$ and $\beta = (2\ 3)$.
For the switching sequence 
$\mathcal{S} = (x, \alpha)(y, \beta) (x, \alpha^{-1})$ we have 
$\Sigma(G^{\mathcal{S}})(xy) = 3$, whereas for the
switching sequence 
$\mathcal{S'} = (x, \alpha)(x, \alpha^{-1}) (y, \beta)$
we have $\Sigma(G^{\mathcal{S'}})(xy) = 1$.

\section{Property $\mathcal{T}_j$}

In this section we introduce a group property which we call \emph{property $\mathcal{T}_j$}.
If the group $\Gamma$ has property $\mathcal{T}_j$ then 
for any edge $xy$ in an $m$-edge coloured graph $G$ there exists
a $\Gamma$-switching sequence such that $xy$ is of 
colour $j$, and the colour of every other
edge of $G$ is unchanged.
It follows that $G$ can be transformed to be 
monochromatic of colour $j$ by changing the
colour of one edge at a time.

We motivate the definition of  property $\mathcal{T}_j$ by considering switching with respect to $S_m$, $m \geq 3$.

\begin{proposition}
Let $G$ be a $m$-edge-coloured graph, where $m \geq 3$, and let $i, j \in \{1, 2, \ldots, m\}$ be such that $i \neq j$.
If $xy \in E_i(G)$, 
then there exists $G^\prime \in [G]_{S_m}$
such that $xy \in E_j(G)$ and $G - xy = G^\prime - xy$.
\label{RecolourEdgeS_m}
\end{proposition}

\begin{proof}
Since $m \geq 3$, 
 for any $k \in \{1, 2, \ldots, m\}$ there exists
 a transposition $(i\ j)$ with $i, j \in \{1, 2, \ldots, m\}\setminus\{k\}$.
 Let $\alpha = (i\ j)$ and $\beta = (j\ k)$.
Consider the $S_m$-switching sequence $(x, \alpha),$ $(y, \beta), (x, \alpha), $ $(y, \beta)$.
This transforms $G$ into $G^\prime$.

The only edges which change colour in the transformation are incident with $x$ or $y$.  
It is given that the edge $xy$ has colour $i$ in $G$.
After the first, second, third and fourth switch, the edge $xy$ has colour $j, k, k, j$, respectively, in the transformed graph.
Any edge $e$ incident with $x$ and not $y$ changes from its colour, $c_e$, to $\alpha(c_e)$ and then back to $c_e$.
Similarly, any edge incident with $y$ has the same colour as in $G$ after switching.
The result follows.
\end{proof}

\begin{corollary}
For $m \geq 3$, two $m$-edge-coloured graphs $G$ and $G'$ are $S_m$-switch equivalent if and only if $\mathit{underlying}(G) \cong \mathit{underlying}(G')$.
\label{S_mSwitchEquiv}
\end{corollary}

Let $\Gamma$ be a subgroup of $S_m$.
For $i, j \in \{1, 2 \ldots, m\}$,
we say $\Gamma$ has \emph{property $\mathcal{T}_{i,j}$} 
if there exist  permutations $\alpha, \beta \in \Gamma$ such
that $\alpha$ maps $i$ to $j$ and fixes some element $k$, and
$\beta$ maps $j$ to $k$.

\begin{proposition}
Let $\Gamma$ be a subgroup of $S_m$ with property $\mathcal{T}_{i,j}$.
If $G$ is a $m$-edge-coloured graph, where $m \geq 3$,
and $xy \in E_i(G)$, 
then there exists $G^\prime \in [G]_{\Gamma}$
such that $xy \in E_j(G)$ and $G - xy = G^\prime - xy$.
\label{RecolourEdgeGamma}
\end{proposition}
\begin{proof}
The switching sequence $(x, \alpha), (y, \beta), (x, \alpha^{-1}), (y, \beta^{-1})$ transforms $G$ to $G'$.
\end{proof}

If there exists $j \in \{1, 2, \ldots, m\}$ such that $\Gamma$ has Property $T_{i,j}$ for all $i \in \{1, 2, \ldots, m\}$, then we say it has
\emph{Property $\mathcal{T}$}. 

\begin{corollary}
Let $\Gamma$ be a subgroup of $S_m$ with property $\mathcal{T}_j$.
Then for any $m$-edge coloured graph $G$ and any colour $j$ there exists
a switching sequence $\mathcal{S}$ such that $G^\mathcal{S}$
is monochromatic of colour $j$.
\label{Monoj}
\end{corollary}

\begin{corollary}
Let $\Gamma$ be a subgroup of $S_m$
with property $\mathcal{T}_j$.
Then two $m$-edge-coloured graphs $G_1$ and $G_2$ are $\Gamma$-switch equivalent if and only if $\mathit{underlying}(G_1) \cong \mathit{underlying}(G_2)$.
\label{S_mSwitchEquiv}
\end{corollary}

\begin{proof}
By definition, two $m$-edge coloured graphs which are
$\Gamma$-switch equivalent 
have isomorphic underlying graphs.

Now suppose $\mathit{underlying}(G_1) \cong \mathit{underlying}(G_2)$.
By Corollary \ref{Monoj}, both $G_1$ and $G_2$ are $\Gamma$-switch equivalent to an $m$-edge-coloured graph which is monochromatic
of colour $j$. Since $\Gamma$-switch equivalence is an equivalence 
relation, $G_1$ and $G_2$ are $\Gamma$-switch equivalent.
\end{proof}

It is easy to see that when $m \geq 4$ the alternating 
group $A_m$ has property $\mathcal{T}j$.

\begin{corollary}
For $m \geq 4$, two $m$-edge-coloured graphs $G$ and $G'$ are $A_m$-switch equivalent if and only if $\mathit{underlying}(G) \cong \mathit{underlying}(G')$.
\end{corollary}

\section{The dihedral group}

For $m \geq 3$ we denote by $D_m$ the group of
permutations of $\{1, 2, \ldots, m\}$
corresponding to symmetries of the regular $m$-gon 
with
vertices $1, 2, \ldots, m$ in cyclic order.
The cases $m$ odd and $m$ even are different. 
We consider the case of odd $m$ first.

\begin{proposition}
For any odd integer $m \geq 3$ and any $j \in \{1, 2, \ldots, m\}$, the group $D_m$ 
has property $\mathcal{T}_{j}$.
\end{proposition}

\begin{proof}
Let $i, j \in \{1, 2, \ldots, m\}$.
We show $D_m$ has property $T_{i, j}$.
There is nothing to prove if $i = j$, so assume $i \neq j$.

Since $m$ is odd, either the least residue of $i - j$ modulo $m$ is even, or the least residue of $j - i$ modulo $m$ is even.
Without loss of generality, the latter holds.
Then there exists $k  \in \{1, 2, \ldots, m-1\}$ such that $j - i \equiv 2k\ (\mathrm{mod}\ m)$, so that $j - k \equiv k - i\ (\mathrm{mod}\ m)$.
Let $\alpha$ be the permutation of $\{1, 2, \ldots, m\}$ which corresponds to flipping the $m$-gon over while fixing vertex $k$.
Then $\alpha$ maps $i$ to $j$ and fixes $k$, so $D_m$ has property $\mathcal{T}_{i,j}$
This completes the proof.
\end{proof}

\begin{corollary}
For odd $m \geq 3$, two $m$-edge-coloured graphs $G$ and $G'$ are $D_m$-switch equivalent if and only if $\mathit{underlying}(G) \cong \mathit{underlying}(G')$.
\end{corollary}

We now consider the case of switching with respect to $D_m$ when $m$ is even.
The following basic facts from group theory will be used.

\begin{observation}
Suppose $m \geq 2$ is even.  Let
$\mathcal{E} = \{2, 4, \ldots, m\}$ and $\mathcal{O} = \{1, 3, \ldots, m-1\}$.
Then,
\begin{enumerate}
\item $\{ \mathcal{O}, \mathcal{E} \}$ is a block system 
for the action of $D_m$ on $\{1, 2, \ldots, m\}$.
\item $\mathit{Stabilizer}(\mathcal{E}) = \mathit{Stabilizer}(\mathcal{O})$ is a normal subgroup of $D_m$, 
\item $D_m / \mathit{Stabilizer}(\mathcal{E}) \cong D_m / \mathit{Stabilizer}(\mathcal{O}) \cong S_2$, and
\item $\mathit{Stabilizer}(\mathcal{E})$ has Property $\mathcal{T}_{i,j}$ for all $i,j \in \mathcal{E}$, and
\item  $\mathit{Stabilizer}(\mathcal{O})$ has Property $\mathcal{T}_{i,j}$ for all $i,j \in \mathcal{O}$.
\end{enumerate}
\label{GroupObs}
\end{observation}

Let $G$ be an $m$-edge-coloured graph, where $m \geq 2$ is an even integer.
The 2-edge-coloured graph $G_2$ is obtained from $\mathit{underlying}(G)$ by assigning each edge $e$ 
colour 1 if  $\Sigma(G)(e) \in \mathcal{O}$, and colour 2 if 
$\Sigma(G)(e) \in \mathcal{E}$.
Notice that this is equivalent to regarding the edge colours of $G_2$ to be $\mathcal{E}$ and $\mathcal{O}$, 
with the colour of an edge of $G_2$ being the name of the block containing the colour of the corresponding edge in $G$.
The colours of the edges of $G_2$ are naturally permuted by 
$D_m / \mathit{Stabilizer}(\mathcal{E})  \cong S_2$.

\begin{theorem}
Let $G$ and  $H$ be $m$-edge-coloured graphs, where $m \geq 2$ is an even integer.
Then $G$ and $H$  are switch equivalent with respect to $D_m$ if and only if
$G_2$ and $H_2$ are switch equivalent with respect to $S_2$.
\label{ThmG2H2}
\end{theorem}

\begin{proof}
Suppose $G$ and $H$  are switch equivalent with respect to $D_m$.
Then there is a $D_m$-switching sequence $\mathcal{S} = (x_1, \pi_1), (x_2, \pi_2), \ldots, (x_t, \pi_t)$ 
that transforms $G$  to $H$.  
By Observation \ref{GroupObs}, each permutation $\pi_i \in D_m$ either maps  $\mathcal{E}$ to $\mathcal{E}$ and $\mathcal{O}$ to $\mathcal{O}$,
or maps $\mathcal{E}$ to $\mathcal{O}$ and vice-versa. 
Let $\mathcal{S}^\prime$ be the subsequence of $\mathcal{S}$ 
consisting of the permutations that map $\mathcal{E}$ to $\mathcal{O}$.
Replacing each of the permutations in this subsequence by the transposition $(1\ 2)$
gives a $S_2$-switching sequence that transforms $G_2$ so it is isomorphic to $H_2$.

Now suppose $G_2$ and $H_2$ are switch equivalent with respect to $S_2$.
Then, $\mathit{underlying}(G_2) \cong \mathit{underlying}(H_2)$. Without loss of generality, $\mathit{underlying}(G_2) = \mathit{underlying}(H_2)$.
Let $\mathcal{A} = (x_1, \sigma_1), (x_2, \sigma_2), \ldots, (x_p, \sigma_p)$ be
a $S_2$-switching sequence  
that transforms $G_2$  to $H_2$.   
Replacing each permutation $\sigma_i \in S_2$ by the $m$-cycle $(1\ 2\ \cdots\ m) \in D_m$ gives a $D_m$-switching sequence that transforms $G$ to a graph $G^\prime$ in which the colour of edge 
belongs to the same block as the corresponding edge $G$.
Since $\mathit{Stabilizer}(\mathcal{E})$
has Property $\mathcal{T}_{j}$ for all $j \in \mathcal{E}$,
and $\mathit{Stabilizer}(\mathcal{O})$
has Property $\mathcal{T}_{j}$ for all $j \in \mathcal{O}$, the $m$-edge-coloured graph $G^\prime$ is 
$D_m$-switch equivalent to $H$ (as in the proof of Proposition \ref{RecolourEdgeGamma} edges other then the one whose colour is intended to change switches from their colour then back again).
\end{proof}

Zaslavsky proved that the  2-edge-coloured graphs $G$ and $H$ with the same underlying graph are 
switch equivalent with respect to $S_2$
if and only if they have the same collection of cycles 
for which the number of edges whose colour is in 
$E_2$ is odd \cite{Zaslavsky}.  Together with Theorem \ref{ThmG2H2}, this yields a similar
result for $D_m$, where $m \geq 2$ is even.

\begin{corollary}
Suppose $m \geq 2$ is even.
Two $m$-edge-coloured graphs $G$ and $H$ with the same underlying graph are  switch equivalent with respect to $D_m$
if and only if  they have  the same collection of cycles 
for which the number of edges whose colour is in
$\mathcal{E}$ is odd.
\label{ZaslavskyTheorem}
\end{corollary}

\section{Colourings and Homomorphisms} 

Recall that a $k$-colouring of a graph $G$ is a function $c: V(G) \to \{1, 2, \ldots, k\}$ such that if $xy \in E(G)$ then $c(x) \neq c(y)$.  A \emph{homomorphism} from a graph $G$ to a graph $H$ is a function $f:V(G) \to V(H)$ such that if $xy \in E(G)$, then $f(x)f(y) \in E(H)$.

Our goal in this section is to present analogues of Theorems \ref{k-colThm} and 
\ref{HNThm} below for $\Gamma$-switchable colourings and homomorphisms when
$\Gamma$ is a group with property $\mathcal{T}_j$ for some $j \in \{1, 2, \ldots, m\}$,
or an even order dihedral group.
Theorems such as these are known as \emph{dichotomy theorems}
because they exhibit a dichotomy for the complexity of a 
particular decision problem.

\begin{theorem}[\cite{GareyJohnson}]
For an integer $k \geq 1$,
the problem of deciding whether a given graph $G$ has a $k$-colouring is 
solvable in polynomial time when $k \leq 2$, and is NP-complete if $k \geq 3$.
\label{k-colThm}
\end{theorem}

\begin{theorem}[\cite{HN}]
If $H$ is a fixed graph then the problem of deciding whether a 
given graph $G$ has a homomorphism to $H$ is solvable in polynomial time if $H$ is
bipartite, and is NP-complete if $H$ is not bipartite.
\label{HNThm}
\end{theorem}

Before introducing colourings and homomorphisms of $m$-edge coloured graphs, we review a connection between these concepts for graphs (without loops). 
Suppose there is a homomorphism, $f$, of the graph $G$ to a graph $H$ on $k$ vertices.
If the vertices of $H$ are regarded as colours, then $f$ is an assignment of these colours to the vertices of $G$ such that adjacent vertices in $G$ are assigned adjacent (hence different) colours.
Thus a $k$-colouring of a graph $G$ can equivalently be defined as a homomorphism
of $G$ to \emph{some} graph $H$ on $k$ vertices.
Defining a $k$-colouring in this way allows the idea of a (vertex) $k$-colouring to be extended to $m$-edge-coloured graphs
(see \cite{NesetrilRaspaud}), 
oriented graphs \cite{Sopena}, and other types of graphs (see \cite{NesetrilRaspaud}).

Let $G$ and $H$ are $m$-edge-coloured graphs.
A \emph{homomorphism} of $G$ to $H$ is a function $f: V(G) \to V(H)$ such that, 
for all $i \in \{1, 2, \ldots, m\}$, if $xy \in E_i(G)$ then $f(x)f(y) \in E_i(H)$.
For an integer $k \geq 1$, a \emph{vertex $k$-colouring} of an $m$-edge-coloured 
graph $G$ is a homomorphism of $G$ to some $m$-edge-coloured graph on $k$ vertices.

Let $\Gamma$ be a subgroup of $S_m$.
A $m$-edge-coloured graph $G$ has a $\Gamma$-switchable homomorphism to an 
$m$-edge-coloured graph $H$ if some $G' \in [G]_\Gamma$ has a homomorphism to $H$, that is, if $G$ can be $\Gamma$-switched so that the transformed graph has
a homomorphism to $H$.
For an integer $k \geq 1$, an $m$-edge-coloured graph $G$ has a 
\emph{$\Gamma$-switchable $k$-colouring} 
if it has a $\Gamma$-switchable homomorphism to some $m$-edge-coloured graph on $k$ vertices.

If follows from the definition that if there exists a $\Gamma$-switchable homomorphism of an 
$m$-edge-coloured graph $G$ to an $m$-edge-coloured graph $H$, then there is a homomorphism from
$\mathit{underlying}(G)$ to $\mathit{underlying}(H)$.  To see that the converse is false, let $\Gamma = S_2$, let $G$ be the 2-edge coloured $K_3$ with two edges of colour 1 and one edge of colour 2 and let $H$ be the 2-edge coloured $K_3$ with two edges of colour 2 and one edge of colour 1.
There is a homomorphism of $\mathit{underlying}(G)$ to $\mathit{underlying}(H)$ but
no $S_2$-switchable homomorphism of $G$ to $H$.

The following theorem from \cite{LMW} is useful because it transforms the problem of deciding whether there is a $\Gamma$-switchable homomorphism of $G$ to $H$ to the problem of deciding the existence of a homomorphism (with no switching) to any $m$-edge coloured graph $\Gamma$ switch equivalent to $H$.

\begin{theorem}[\cite{LMW}]
Let $G$ and $H$ be $m$-edge-coloured graphs and let $\Gamma$ be a subgroup of $S_m$.
Then there is a $\Gamma$-switchable homomorphism of $G$ to $H$ if and only if,
for all $H' \in [H]_\Gamma$ there exists $G' \in [G]_\Gamma$
such that there is a homomorphism of $G'$ to $H'$.
\label{SwitchH}
\end{theorem}

\begin{theorem}
Let $\Gamma$ be a subgroup of $S_m$ that has Property $\mathcal{T}_j$ for some $j\in \{1, 2, \ldots, m\}$,
and let $k \geq 1$ be an integer.
If $k \leq 2$ then the problem of deciding whether a given $m$-edge-coloured graph has a
$\Gamma$-switchable $k$-colouring is solvable in polynomial time.
If $k \geq 3$ then the problem of deciding whether a given $m$-edge-coloured graph has a 
$\Gamma$-switchable $k$-colouring is NP-complete.
\end{theorem}
\begin{proof}
By Corollary \ref{S_mSwitchEquiv} every $m$-edge-coloured graph $F$  is 
$\Gamma$-switch equivalent to an
$m$-edge-coloured graph $F'$ which is monochromatic of colour $j$.
Thus by Theorem \ref{SwitchH}  there is a $\Gamma$-switchable homomorphism
of $G$ to an $m$-edge-coloured graph $H$ on $k$ vertices 
if and only if there is a homomorphism of $G'$ to $H'$,
if and only if there is a
homomorphism of $\mathit{underlying}(G)$ to $\mathit{underlying}(H)$,
if and only if  $\mathit{underlying}(G)$ has a $k$-colouring.
The result now follows from from Theorem \ref{k-colThm}.
\end{proof}

The proof of the corresponding result for $\Gamma$-switchable homomorphisms is virtually identical.

\begin{theorem}
Let $\Gamma$ be a subgroup of $S_m$ that has Property $\mathcal{T}_j$ for some $j\in \{1, 2, \ldots, m\}$.
Let $H$ be a fixed $m$-edge-coloured graph.
If $H$ is bipartite the problem of deciding whether a given $m$-edge-coloured graph has a 
$\Gamma$-switchable homomorphism to $H$ is solvable in polynomial time.
If $H$ is not bipartite then the problem of deciding whether a given $m$-edge-coloured graph has a 
$\Gamma$-switchable homomorphism to $H$ is NP-complete.
\end{theorem}
\begin{proof}
By Corollary \ref{S_mSwitchEquiv}, the $m$-edge-coloured graphs $G$ and $H$ are 
$\Gamma$-switch equivalent to 
$m$-edge-coloured graphs $G'$ and $H'$ which are monochromatic of colour $j$.
It follows from Theorem \ref{SwitchH} that there is a $\Gamma$-switchable homomorphism
of $G$ to $H$ if and only if there is a homomorphism of $G'$ to $H'$, if and only if there is a
homomorphism of $\mathit{underlying}(G)$ to $\mathit{underlying}(H)$.
The result now follows from Theorem \ref{HNThm}.
\end{proof}

We have found dichotomy theorems for the complexity of
the $\Gamma$-switchable $k$-colouring problem and the problem of deciding whether there exists a $\Gamma$-switchable homomorphism to a fixed $m$-edge coloured graph $H$ when $\Gamma$ is one of $S_m, m\geq 3$; $A_m, m\geq 4$; $D_m, m \geq 2$ and odd; any other group with property $T_j$ for some $j$.
Finally, we consider dihedral groups of even order.
The following theorem will be useful.

\begin{theorem}[\cite{BFHN}]
Let $H$ be a 2-edge-coloured graph.
If there is a $S_2$-switchable homomorphism of $H$ to a monochromatic $K_2$, 
then the problem of deciding whether a given $2$-edge-coloured graph $G$
has a $S_2$-switchable homomorphism to $H$ is solvable in polynomial time.
If there is no $S_2$-switchable homomorphism  of $H$ to a monochromatic $K_2$, then
the problem of deciding whether a given $2$-edge-coloured graph $G$
has a $S_2$-switchable homomorphism to $H$ is NP-complete.
\label{BFHNThm}
\end{theorem}

As an aside, we note that it is easy to test whether a 2-edge coloured graph is $S_2$-switch equivalent to a monochromatic $K_2$.
By \cite{BrewsterGraves} such a $S_2$-switchable homomorphism
exists if and only if there is a homomorphism (without switching) of $H$ to a 4-cycle where the edge colours alternate.
The latter condition can be tested in polynomial time.  
Without loss of generality $H$ is connected.
By symmetry the image of any vertex 
can be chosen without loss of generality.  Since each vertex is adjacent with exactly one edge of each colour, there is only one choice to extend the mapping to a neighbouring vertex.  Successively doing so either leads to the desired homomorphism or to a contradiction in which some vertex is forced to have two different images.

\begin{theorem}
Let $k \geq 1$ be an integer, and let $m \geq 2$ be an even integer.
The problem of deciding whether a given $m$-edge-coloured graph 
has a $D_m$-switchable $k$-colouring is solvable in polynomial time if $k \leq 2$ and
is NP-complete if $k \geq 3$.
\end{theorem}
\begin{proof}
Let $G$ be an $m$-edge coloured graph.  It is clear that $G$ has a $D_m$-switchable 1-colouring if
and only if it has no edges.

Suppose that $k=2$.
By definition,  $G$ has a $D_m$-switchable 2-colouring if and only if there exists $j$ such that it has a $D_m$-switchable homomorphism to a $K_2$ of colour $j$.
Thus, by Theorem \ref{SwitchH}, $G$ has a $D_m$-switchable 2-colouring if and only if it is bipartite and there exists
$G' \in [G]_{D_m}$ such that $G'$ is monochromatic of colour $j$.

Without loss of generality $j$ is odd.
By Theorem \ref{ThmG2H2} the $m$-edge coloured graph $G'$ exists
if and only if $G_2$ (as in Theorem \ref{ThmG2H2})  is $S_2$-switch equivalent to $G'_2$.
Since $G'_2$ is bipartite and switchable homomorphisms compose \cite{LMW}, this is equivalent to $G_2$ having a $S_2$-switchable homomorphism 
to a $K_2$ of colour 1, which is decidable in polynomial time by Theorem \ref{BFHNThm}.

Now suppose $k \geq 3$.
The transformation is from the problem of deciding whether a given graph $G$ has a $k$-colouring.
Suppose a graph $G$ is given.
We claim that $G$ has a $k$-colouring if and only if the $m$-edge-coloured graph $G'$ 
with $\mathit{underlying}(G') = G \cup K_k$ and
which is monochromatic
of colour $j$  has a $D_m$-switchable $k$-colouring.
Clearly $G'$ can be constructed in polynomial time.

We now show that $G$ has a $k$-colouring if and only if $G'$ has a $D_m$-switchable $k$-colouring.

Suppose $G'$ has a $D_m$-switchable $k$-colouring.
By definition, such a mapping is a $k$-colouring of $\mathit{underlying}(G')$.
Since $G$ is a subgraph of $\mathit{underlying}(G')$, it follows that
$G$ is $k$-colourable.

Now suppose that $G$ has a $k$-colouring.
Then there is a homomorphism of $G \cup K_k$ to $K_k$.
Therefore there is a homomorphism of $G'$ to a $K_k$ which is monochromatic of colour $j$.
Thus $G'$ has a $\Gamma$-switchable $k$-colouring. 

It now follows that $\Gamma$-switchable $k$-colouring is NP-complete.
\end{proof}

The proof of the following lemma is very similar to the proof of Theorem \ref{ThmG2H2}.

\begin{lemma}
Let $G$ and $H$ be $m$-edge coloured graphs, and $m \geq 2$ be an even integer.
There is a $D_m$-switchable homomorphism of $G$ to $H$ if and only if
there is a $S_2$-switchable homomorphism of $G_2$ to $S_2$.
\label{LemG2toH2}
\end{lemma}
\begin{proof}
Suppose first that there is a $D_m$-switchable homomorphism of $G$ to $H$.
Then there is a $D_m$-switching sequence $\mathcal{S} = (x_1, \pi_1), (x_2, \pi_2), \ldots, (x_t, \pi_t)$ 
that transforms $G$  to $G' \in [G]_{D_m}$ for which there is a homomorphism of $G'$ to $H$.
By Observation \ref{GroupObs}, each permutation $\pi_i \in D_m$ either maps  $\mathcal{E}$ to $\mathcal{E}$ and $\mathcal{O}$ to $\mathcal{O}$,
or maps $\mathcal{E}$ to $\mathcal{O}$ and vice-versa. 
Let $\mathcal{S}^\prime$ be the subsequensce of $\mathcal{S}$ 
consisting of the permutations that map $\mathcal{E}$ to $\mathcal{O}$.
Replacing each of the permutations in this subsequence by the transposition $(1\ 2)$
gives a $S_2$-switching sequence that transforms $G_2$ to 
a 2-edge coloured graph $G_2'$ that has a homomorphism to $H_2$.  Therefore there is a $S_2$-switchable homomorphism of $G_2$ to $H_2$.

Now suppose there is a $S_2$-switchable homomorphism of $G_2$ to $H_2$.

We claim that $G$ and $H$ are $D_m$-switch equivalent to $G_2$ and $H_2$ (considered as $m$-edge coloured graphs), respectively.  
Since $\mathit{Stabilizer}(\mathcal{E})$
has property $\mathcal{T}_{j}$ for all $j \in \mathcal{E}$, any edge of $G$ 
whose colour is in $\mathcal{E}$ can be $D_m$-switched to have colour 2. As in the proof of  Proposition \ref{RecolourEdgeGamma}, edges other then the one whose colour is intended to change switch from their colour then back again.  Similarly, any edge of $G$ 
whose colour is in $\mathcal{O}$ can be $D_m$-switched to have colour 1.
Thus $G$ is $D_m$-switch equivalent to $G_2$.
Similarly $H$ is $D_m$-switch equivalent to $H_2$,
and the claim is proved.

The same function which is a homomorphism of the 2-edge coloured graph $G_2$ to the 2-edge coloured graph $H_2$ is also a homomorphism of the $m$-edge coloured graph $G_2$ to the $m$-edge coloured graph $H_2$.  
It now follows from Theorem \ref{SwitchH} that there is a $D_m$-switchable
homomorphism of $G$ to $H$.
\end{proof}

\begin{theorem}
Let  $H$ be an $m$-edge coloured graph, and $m \geq 2$ be an even integer.
If there is a homomorphism of $H_2$ to a monochromatic $K_2$, 
then the problem of deciding whether a given $m$-edge-coloured graph $G$
has a $D_m$-switchable homomorphism to $H$ is solvable in polynomial time.
If there is no homomorphism  of $H_2$ to a monochromatic $K_2$,
the problem of deciding whether a given $m$-edge-coloured graph $G$
has a $D_m$-switchable homomorphism to $H$ is NP-complete.
\end{theorem}

\begin{proof}
Let  $H$ be an $m$-edge coloured graph, and $m \geq 2$ be an even integer.

Suppose first  there is a homomorphism of (the 2-edge coloured graph) $H_2$ to a monochromatic $K_2$. Let $G$ be an $m$-edge coloured graph.
Then, by Theorem \ref{BFHNThm}, it can be decided in polynomial time whether 
$G_2$ has a  $S_2$-switchable homomorphism to $H_2$.
Since $G_2$ can be constructed in polynomial time, Lemma \ref{LemG2toH2} implies  it can 
be decided in polynomial time whether $G$ has a 
$D_m$-switchable homomorphism to $H$.

Now suppose there is no homomorphism  of (the 2-edge coloured graph) $H_2$ to a monochromatic $K_2$.
We want to show  the problem of deciding whether a given $m$-edge coloured graph
has a $D_m$-switchable homomorphism to $H$ is NP-complete.
The transformation is from the problem of deciding whether a given 
2-edge coloured graph $F$ has a $S_2$-switchable homomorphism to $H_2$, which is NP-complete by Theorem \ref{BFHNThm}.
 
Suppose such a 2-edge-coloured graph $F$ is given.
The transformed instance the problem is the $m$-edge
coloured graph $F'$ in which every edge has the 
same colour as in $F$.
Then $F'_2 = F$, and the result follows from Lemma \ref{LemG2toH2}.
\end{proof}

\vfill

\smallskip\noindent
Chris Duffy,
\textsc{School of Mathematics and Statistics, University of Melbourne, Melbourne, Australia}\\
{\tt christopher.duffy@unimelb.edu.au}\\

\noindent
Gary MacGillivray,
\textsc{Mathematics and Statistics, University of Victoria, Victoria, Canada}\\
{\tt gmacgill@math.uvic.ca}\\

\smallskip \noindent
Ben Tremblay,
\textsc{Combinatorics and Optimization, University of Waterloo, Waterloo, Canada}, \\
{\tt ben.tremblay@uwaterloo.ca}\\

\bigskip\noindent
{\bf Acknowledgement}.  Research of the first two authors supported by NSERC.

\end{document}